\documentclass{amsart}

\usepackage{amsmath,amsthm,amssymb,amsfonts}
\usepackage{enumerate}
\usepackage{mathrsfs}
\usepackage[cmtip,all]{xy}
\usepackage{tikz}
\usepackage{tikz-cd}

\newtheorem{theorem}{Theorem}[section]
\newtheorem{lemma}[theorem]{Lemma}
\newtheorem{proposition}[theorem]{Proposition}

\theoremstyle{definition}
\newtheorem{definition}[theorem]{Definition}
\newtheorem{example}[theorem]{Example}

\theoremstyle{remark}

\numberwithin{equation}{section}

\let\smash=\wedge
\let\iso=\cong
\let\tensor=\otimes
\let\directsum=\oplus

\newcommand{\unit}{\mathbf{1}}
\newcommand{\SH}{\mathbf{SH}}

\newcommand{\unitkq}{\mathrm{u}}

\newcommand{\GW}{\mathbf{GW}}
\newcommand{\I}{\mathbf{I}}

\newcommand{\bideg}{(\star)}
\newcommand{\hyper}{\mathsf{h}}

\newcommand{\moore}[1]{\mathsf{C}_{#1}}

\newcommand{\Hom}{{\operatorname{Hom}}}
\newcommand{\Spec}{{\operatorname{Spec}}}

\newcommand{\KQ}{\mathsf{KQ}}
\newcommand{\KW}{\mathsf{KW}}
\newcommand{\ckw}{\mathsf{cKW}}

\newcommand{\KGL}{\mathsf{KGL}}

\newcommand{\kq}{\mathsf{kq}}
\newcommand{\kw}{\mathsf{kw}}

\newcommand{\C}{\mathsf{C}}
\newcommand{\D}{\mathsf{D}}
\newcommand{\E}{\mathsf{E}}

\newcommand{\F}{\mathsf{F}}
\newcommand{\GG}{\mathsf{G}}

\newcommand{\dd}{\mathsf{d}}

\newcommand{\s}{\mathsf{s}}
\newcommand{\f}{\mathsf{f}}

\newcommand{\Mod}{\mathrm{mod}\,}
\newcommand{\const}{\mathrm{const}}

\newcommand{\Mil}{\mathsf{M}}
\newcommand{\KMil}{\mathbf{K}^\mathsf{M}}
\newcommand{\kmil}{\mathbf{k}^\mathsf{M}}
\newcommand{\KMW}{\mathbf{K}^\mathsf{MW}}

\newcommand{\sphere}{\mathbb{S}}
\newcommand{\NN}{\mathbb{N}}
\newcommand{\ZZ}{\mathbb{Z}}
\newcommand{\FF}{\mathbb{F}}
\newcommand{\pr}{{\mathrm{pr}}}
\newcommand{\id}{{\mathrm{id}}}

\newcommand{\inc}{{\mathrm{inc}}}

\newcommand{\A}{\mathbf{A}}

\newcommand{\MZ}{\mathbf{M}\mathbb{Z}}

\newcommand{\RR}{\mathbb{R}}

\newcommand{\SL}{\mathbf{SL}}
\newcommand{\Q}{\mathsf{Q}}
\newcommand{\QQ}{\mathbb{Q}}
\newcommand{\Sq}{\mathsf{Sq}}

\newcommand{\Top}{\mathsf{top}}

\newcommand{\Sm}{\mathbf{Sm}}

\begin{document}



\title{Remarks on motivic Moore spectra}
\author{Oliver R\"ondigs}
\address{Institut f\"ur Mathematik, Universit\"at Osnabr\"uck, Germany}
\email{oliver.roendigs@uni-osnabrueck.de}
\thanks{This work was supported through DFG grants within the SPP 1786 ``Homotopy theory and algebraic geometry'', and a guest professorship at the University of Oslo.}


\date{\today}
\begin{abstract}
  The term ``motivic Moore spectrum'' refers to a cone of
  an element $\alpha\colon \Sigma^{s,w}\unit\to \unit$ in the
  motivic stable homotopy groups of spheres. Homotopy groups,
  multiplicative structures, and Voevodsky's slice spectral
  sequence are discussed for motivic Moore spectra.
\end{abstract}

\maketitle

\section{Introduction}
\label{section:introduction}

Let $R$ be a ring and $a\in R$ an element, generating a two-sided
ideal $(a)\subset R$. The projection
onto the quotient $R\to R/a:=R/(a)$ is then a ring homomorphism.
If $R$ is commutative, then so is $R/a$.
In homotopy theory, the situation is more subtle. The 
topological sphere spectrum $\sphere$ is the unit in a closed symmetric
monoidal category modeling the stable homotopy category, and in particular
a commutative monoid. Given any 
endomorphism $a\colon \sphere\to \sphere$, the homotopy-theoretic quotient
is almost never a commutative monoid. The first instance occurs when $a=2$:
The homotopy-theoretic quotient $\sphere/2$, also known as the Moore spectrum
for the group $\ZZ/2$, satisfies $\pi_2\sphere/2\iso \ZZ/4$ by 
\cite{barratt.track2}, and in particular admits no unital multiplicative
structure.
Even if a unital multiplicative structure exists on $\sphere/n$ for $n>2$, 
its associativity or commutativity in the stable homotopy category is not
automatic. See \cite{oka.mult-moore} and the references therein for details.

Within motivic or $\mathbf{A}^1$-homotopy theory, Moore spectra have appeared
for example in \cite{roendigs-oestvaer.rigidity}, \cite{gheorghe}, and
\cite{mantovani.thesis}.
As the structure of the endomorphisms of the motivic sphere spectrum
is much richer, also the notion of Moore spectra should be more
sophisticated. The degree zero part of these endomorphisms over a field
is the Milnor-Witt $K$-theory graded by weight \cite{morelmotivicpi0},
and the weight zero part of that is the Grothendieck-Witt ring of the
field. Since neither is a principal ideal domain in general, and
usually far from Noetherian, Moore
spectra with respect to ideals instead of single elements are
more sensible. See for example \cite[Remark 1.4]{mantovani.thesis}. 
Nevertheless, an elementary approach is chosen here, which still
suffices to illustrate a few interesting phenomena. 
More precisely, multiplicative structures on motivic Moore spectra
-- whose existence may depend on the ground field -- are discussed
in Section~\ref{sec:multiplications}, based to some extent on
results concerning Toda brackets listed in Section~\ref{sec:toda-brackets}.
These in turn rely on some preliminaries on the few first stable
stems of the motivic sphere spectrum, to be discussed in Section~\ref{sec:preliminaries-pi_1}, which
follows and partly expands~\cite{rso.oneline}. The article closes
with some results on slices and slice differentials for special
motivic Moore spectra in Section~\ref{sec:slices-moore-spectra}.
These results may be used for slice spectral sequence
computations of homotopy groups
of motivic Moore spectra. A noteworthy feature in 
comparison with corresponding slice spectral
sequence computations for the motivic sphere spectrum is the absence of 
motivic cohomology groups with integral coefficients;
motivic cohomology groups with finite coefficients are understood much better.

\section{Preliminaries on $\pi_0$, $\pi_1$, and $\pi_2$}
\label{sec:preliminaries-pi_1}

Determining the existence of multiplications or pairings on Moore
spectra requires information about stable homotopy groups of motivic
spheres. Let $\pi_{s,w} \E$ denote the abelian group 
$[\Sigma^{s,w}\unit,\E]$, where $\E$ is a motivic spectrum and
$\unit$ is the motivic sphere spectrum. Set $\pi_{s+(w)}\E:=\pi_{s+w,w}\E$, 
and let
\[ \pi_{s+\bideg}\E = \bigoplus_{w\in \ZZ} \pi_{s+w,w}\E \]
denote the direct sum, considered as a $\ZZ$-graded module over the
$\ZZ$-graded ring $\pi_{0+\bideg}\unit$. The notation
$\pi_{s-\bideg}\E:=\pi_{s+(-\star)}\E$ will be used frequently.
The strictly 
$\A^1$-invariant sheaf obtained as the associated Nisnevich sheaf
of $U\mapsto \pi_{s,w} \E_U$ for $U\in \Sm_F$ is denoted
$\underline{\pi}_{s,w}\E$, which gives rise to $\underline{\pi}_{s+\bideg}\E$.
See \cite{morelmotivicpi0} for the following statement.

\begin{theorem}[Morel]\label{thm:zeroline}
  Let $F$ be a field. Then $\pi_{0-\bideg}\unit$ is the Milnor-Witt
  $K$-theory of $F$.
\end{theorem}

The Milnor-Witt $K$-theory of $F$ is denoted $\KMW(F)$, or
simply $\KMW$, following the convention that the base field
or scheme may be ignored in the notation. Its generators are
denoted $\eta\in \KMW_{-1}=\pi_{1,1}\unit$ and 
$[u]\in \KMW_1(F)=\pi_{-1,-1}\unit_F$ for every unit $u\in F^\times$.
The abbreviations
\begin{align*}
  \langle u\rangle & := 1+\eta[u]\in \KMW_0(F) \\ 
  \langle u_1,\dotsc,u_m\rangle & : =\langle u_1\rangle+\dotsm +\langle u_m\rangle
  \in \KMW_0(F) \\  
  \varepsilon&:=-\langle-1\rangle \\
  \hyper& :=\langle 1,-1\rangle =1-\varepsilon
\end{align*}
for units $u,u_1,\dotsc,u_m\in F^\times$ will be convenient.
Under the
identification of $\KMW_{0}(F)$ with the Grothendieck-Witt ring $\GW(F)$
of $F$, the element $\langle u_1,\dotsc,u_m\rangle$ corresponds to the
quadratic form given by the appropriate diagonal matrix. Milnor 
$K$-theory \cite{milnor.k-quadratic} is expressed
as the quotient $\KMil_\star \iso \KMW_\star/(\eta)$. Set
$\kmil_\star:=\KMil_\star/2$.
Theorem~\ref{thm:zeroline} implies that for every motivic spectrum
$\E$ and for every integer $s$, $\pi_{s+\bideg}\E$ is a graded
$\KMW$-module. As a first instance besides the motivic sphere
spectrum $\unit$, consider 
the very effective cover $\kq\to \KQ$
of the motivic spectrum representing hermitian $K$-theory
\cite{aro.kq}, \cite{bachmann.generalized}.
Using $\kq$ instead of 
the effective cover $\f_{0}\KQ\to \KQ$ 
leads to a slight improvement on the 
computation~\cite[Theorem 5.5]{rso.oneline}.

\begin{theorem}[R\"ondigs-Spitzweck-{\O}stv{\ae}r]\label{thm:oneline}
  Let $F$ be a field of exponential characteristic $e\neq 2$. The
  unit map $\unit\to \kq$ induces an isomorphism
  $\pi_{0+\bideg}\unit \to \pi_{0+\bideg}\kq$, and a surjection
  $\pi_{1+\bideg}\unit\to \pi_{1+\bideg}\kq$ whose kernel coincides
  with $\KMil_{2-\star}/24$ after inverting $e$. 
\end{theorem}

Voevodsky's slice filtration $\{\f_{q}\E\to\E\}_{q\in \ZZ}$
allows to be more precise, and in particular
to describe the $\KMW$-module structure.
Let $\MZ$ be Voevdosky's integral motivic Eilenberg-MacLane spectrum
representing motivic cohomology, and let $\MZ/2$ be the
version with coefficients in $\ZZ/2$. Moreover, set for $k$ a natural number 
$H^{\star-k,\star}:=\pi_{k-\bideg}\MZ$ and $h^{\star-k,\star}:=\pi_{k-\bideg}\MZ/2$.
Note the $\KMW$-module isomorphism $h^{\star-k,\star}\iso \kmil_{\star-k}$
given by multiplication with $\tau^k$, where 
\[ \tau=-1\in h^{0,1} = \{1,-1\} = \ker \bigl(H^{1,1} \xrightarrow{x\mapsto x^2} H^{1,1}\bigr)\] 
is the unique nontrivial element. The $\KMW$-module $\pi_{1-\bideg}\kq$ is
an extension of the $\KMW$-module
$H^{\star-1,\star}=\pi_{1-\bideg}\s_0\kq$ (on which $\eta$ operates trivially)
and the $\KMW$-module given by the image of 
$\pi_{1-\bideg}\f_1\kq$ in $\pi_{1-\bideg}\kq$. 


\begin{lemma}\label{lem:pi1f1kq}
  The $\KMW$-module $\pi_{1-\bideg}\f_1\kq$ is 
  generated by the image of $\eta_\Top$ under the 
  unit map $\unitkq\colon \unit\to\kq$, and has the presentation
  \[ \KMW/(2,\eta^2)\iso \pi_{1-\bideg}\f_1\kq.\]
\end{lemma}

\begin{proof}  
  The determination of the relevant
  part of the slice spectral sequence for $\kq$ given in 
  \cite[Proposition 27]{aro.kq} implies that
  the $\KMW$-module $\pi_{1-\bideg}\f_1\kq$ is an extension of the $\KMW$-module
  $h^{\star,1+\star}=\pi_{1-\bideg}\s_1\kq$ (on which $\eta$ operates trivially)
  and the $\KMW$-module $h^{1+\star,2+\star}\!/\Sq^2h^{\star-1,1+\star}$
  (on which $\eta$ operates trivially as well).
  The $\KMil$-module $h^{\star,1+\star}=\pi_{1-\bideg}\s_1\kq$
  is generated by the image of $\eta_\Top \in \pi_{1,0}\unit$,
  and the $\KMil$-module $h^{1+\star,2+\star}\!/\Sq^2h^{\star-1,1+\star}$
  is generated by the image of $\eta\eta_\Top \in \pi_{2,1}\unit$. 
  Since $2\eta_\Top=0$, the extension describing $\pi_{1-\bideg}\f_1\kq$
  splits in every degree as a short exact sequence of abelian groups.
  In other words, the surjection $\KMW\to \pi_{1-\bideg}\kq$
  factors through a surjection $\KMW/2\to \pi_{1-\bideg}\f_1\kq$.
  The relation $0=\eta^2\unitkq(\eta_\Top)\in \pi_{3,2}\f_1\kq$
  follows from the slice spectral sequence 
  computation (which even gives $\pi_{3,2}\kq=0$).
  
  In order to show that the obtained surjection
  $\KMW/(2,\eta^2)\to \pi_{1-\bideg}\f_1\kq$ is an
  isomorphism, observe that it fits into a natural transformation
  \begin{equation}\label{eq:slice-pi1f1kq}
    \begin{tikzcd}
      \eta \KMW/(2,\eta^2) \ar[r]\ar[d] & 
      \KMW/(2,\eta^2) \ar[r] \ar[d] &
      \KMW/(2,\eta) \ar[d] \\
      h^{1+\star,2+\star}\!/\Sq^2h^{\star-1,1+\star} \ar[r] &
      \pi_{1-\bideg}\f_1\kq \ar[r] & 
      \pi_{1-\bideg}\s_1\kq 
    \end{tikzcd}
  \end{equation}
  of short exact sequences, where the outer vertical morphisms are
  isomorphisms. This implies the result. 
\end{proof}

The difference between $\pi_{1-\bideg}\f_1\kq$ and its image
$\f_1\pi_{1-\bideg}\kq$ in $\pi_{1-\bideg}\kq$ is given by the
image of 
$\pi_{2-\bideg}\s_0\kq\iso H^{\star-2,\star}\to \pi_{1-\bideg}\f_1\kq$. 
Since $2\eta_\Top=0$, this map factors over $H^{\star-2,\star}\!/2$.
The short exact sequence displayed in (\ref{eq:slice-pi1f1kq})
then induces a long exact sequence
\begin{small}
\[\Hom_{\kmil}(H^{\star-2,\star}\!/2,\kmil_{\star+1}/\rho^2)
\to \Hom_{\KMW}(H^{\star-2,\star}\!/2,\pi_{1-\bideg}\f_1\kq)
\to \Hom_{\kmil}(H^{\star-2,\star}\!/2,\kmil) \to \dotsm \]
\end{small}
in which the first group is zero for any finite field or
any number field. Here $\rho$ is the class of $-1$ in $h^{1,1}$. 
The homomorphism in question maps to
the restriction of the motivic Steenrod square $\Sq^2$ in the abelian group 
$\Hom_{\kmil}(H^{\star-2,\star}\!/2,h^{\star,\star+1}\iso\kmil_{\star})$, 
and is thereby uniquely determined for prime fields. 

Following a
specific request, a probably well-known identification,
in which $\mathcal{O}^\times$ denotes the sheaf of units, can be
derived from Lemma~\ref{lem:pi1f1kq}. In principle, any
sequence of strictly $\A^1$-invariant sheaves which
is exact when evaluated on fields is already exact, 
by a theorem of Morel. This applies in particular
to Theorem~\ref{thm:oneline}, Lemma~\ref{lem:pi1f1kq}, and
Theorem~\ref{thm:pi1f1unit}. However, the following case can be
proved directly instead.

\begin{proposition}\label{prop:pi10KQ}
  Let $F$ be a field of characteristic not two. The 
  sheaf $\underline{\pi}_{1,0}\KQ$ is isomorphic to the sheaf 
  $\mathcal{O}^\times\!/2\times \ZZ/2$.
\end{proposition}

\begin{proof}
  By construction, the canonical maps induce 
  isomorphisms $\underline{\pi}_{1,0}\f_1\kq \iso 
  \underline{\pi}_{1,0}\kq \iso \underline{\pi}_{1,0}\KQ$.
  This uses the vanishing $\underline{\pi}_{1,0}\s_0\kq = 
  \underline{\pi}_{2,0}\s_0\kq = 0$, since $\s_0\kq =\MZ$
  and motivic cohomology of smooth schemes
  in weight zero is concentrated
  in degree zero. The homotopy cofiber sequence
  \[ \f_2 \kq \to \f_1 \kq \to \s_1 \kq \] 
  induces, for every smooth connected $F$-scheme $U$, 
  a short exact sequence
  \[ 0=\pi_{2,0}\s_1\kq(U) \to \pi_{1,0}\f_2\kq (U)
  \to \pi_{1,0}\f_1\kq(U) \to \pi_{1,0}\s_1\kq(U) =h^{0,1}\!(U) =\ZZ/2 \to 0\]
  which splits naturally via the map $\ZZ/2=\pi_{1}\sphere\to \pi_{1,0}\f_1\kq$.
  In case $U$ is (the spectrum of) a field extension, 
  this sequence appears in diagram~(\ref{eq:slice-pi1f1kq}).
  Already if $U$ is an essentially smooth local $F$-scheme, then 
  $\pi_{1,0}\f_2\kq (U) \iso \pi_{1,0}\s_2\kq (U) \iso h^{1,2}(U)$
  which supplies the result on the level of Nisnevich sheaves.
\end{proof}

\begin{theorem}\label{thm:pi1f1unit}
  Let $\nu\in \pi_{3,2}\unit$ be the second algebraic Hopf map,
  obtained by the Hopf construction on $\SL_2$, and 
  let $\eta_\Top\in \pi_{1,0}\unit$ be the first topological
  Hopf map.
  The $\KMW$-module map 
  \[ \KMW_{2+\star}\directsum \KMW \to \pi_{1-\bideg}\f_1\unit \]
  sending $(a,b)$ to $a\cdot \nu + b\cdot \eta_\Top$ induces
  an isomorphism 
  \[ \KMW_{2+\star}\{\nu\}\directsum \KMW\{\eta_\Top\}/
  (\eta\nu,2\eta_{\Top},\eta^2\eta_{\Top}-12\nu)\iso \pi_{1-\bideg}\f_1\unit\]
  after inverting the exponential characteristic.
\end{theorem}

\begin{proof}  
  Observe first that $\nu$ naturally lifts to $\pi_{3,2}\f_2\unit$, hence
  defines also an element $\nu\in \pi_{3,2}\f_1\unit$. 
  The element $\eta_{\Top}$ has image 
  $0\in \pi_{1,0}\s_0\unit=\pi_{1,0}\MZ=H^{-1,0}$, and hence
  also lifts to $\pi_{1,0}\f_1\unit$ -- even uniquely, since
  $\pi_{2,0}\s_0\unit=\pi_{2,0}\MZ=H^{-2,0}=0$. Thus 
  the $\KMW$-module map 
  \[ \KMW_{2+\star}\directsum \KMW \to \pi_{1-\bideg}\f_1\unit \]
  sending $(a,b)$ to $a\cdot \nu + b\cdot \eta_\Top$ is well-defined
  over any base scheme in which motivic cohomology in weight zero
  vanishes in degree $-1$.
  Certain relations hold before inverting the exponential characteristic.
  The relation $\eta\nu=0$ holds by \cite{dugger-isaksen.hopf}
  over $\Spec(\ZZ)$. Also $2\eta_{\Top}=0$,
  again over any base scheme. If $2$ is invertible in the base scheme,
  $\unitkq(\nu)=0$, where $\unitkq\colon\unit\to\kq$ is the unit map
  figuring in Lemma~\ref{lem:pi1f1kq}.

  Assuming that the exponential characteristic is now implicitly inverted,
  the exact sequence from Theorem~\ref{thm:oneline}
  implies that $\pi_{3,2}\unit=\ZZ/24$, generated by $\nu$.
  Either by reference to complex realization
  or by the multiplicative structure of the slice spectral sequence
  (more precisely, only the effect of multiplying
  with the first algebraic Hopf map $\eta$),
  the element $\eta^2\eta_{\Top}\in \pi_{3,2}\unit$ is the unique
  nontrivial element of order two, hence $\eta^2\eta_{\Top}=12\nu=6 \hyper \nu$.
  Thus there is an induced $\KMW$-module homomorphism
  \[ \KMW_{2+\star}\{\nu\}\directsum \KMW\{\eta_\Top\}/
  (\eta\nu,2\eta_{\Top},\eta^2\eta_{\Top}-12\nu)\to \pi_{1-\bideg}\f_1\unit.\]
  The $\KMW$-module homomorphism 
  $\pi_{1-\bideg}\f_1\unit\to\pi_{1-\bideg}\f_1\kq$ induced by the
  unit map $\unitkq\colon\unit\to\kq$ then factors over the
  isomorphism
  \[ \KMW/(2,\eta^2)\iso \pi_{1-\bideg}\f_1\kq\]
  from Lemma~\ref{lem:pi1f1kq}. There results a commutative diagram
  \begin{equation}\label{eq:slice-pi1f1unit}
    \begin{tikzcd}
      \frac{\KMW_{2+\star}\{\nu\}}{(24\nu,\eta\nu)} \ar[r]\ar[d] & 
      \frac{\KMW_{2+\star}\{\nu\}\directsum \KMW\{\eta_\Top\}}
      {(\eta\nu,2\eta_{\Top},\eta^2\eta_{\Top}-12\nu)} \ar[r] \ar[d] &
      \frac{\KMW\{\eta_{\Top}\}}{(2\eta_{\Top},\eta^2\eta_{\Top})} \ar[d] \\
      \KMil_{2+\star}/24 \ar[r] &
      \pi_{1-\bideg}\f_1\unit \ar[r] & 
      \pi_{1-\bideg}\f_1\kq 
    \end{tikzcd}
  \end{equation}
  of short exact sequences, where the upper horizontal row is
  short exact by direct computation. Note that the relations
  $2\eta_{\Top}=0$ and $\eta^2\eta_{\Top}=12\nu$ imply $24\nu=0$.
  The exactness of the lower horizontal row follows from the
  exactness of the sequence in Theorem~\ref{thm:oneline}
  and the fact that the unit $\unitkq\colon \unit \to \kq$
  induces an isomorphism on zero slices.
  The right vertical map in diagram~(\ref{eq:slice-pi1f1unit})
  is an isomorphism by Lemma~\ref{lem:pi1f1kq}, and the
  left vertical map in diagram~(\ref{eq:slice-pi1f1unit})
  is an isomorphism by direct inspection, whence the result.
\end{proof}

As in the case of $\kq$, the 
difference between $\pi_{1-\bideg}\f_1\unit$ and its image
$\f_1\pi_{1-\bideg}\unit$ in $\pi_{1-\bideg}\unit$ is given by the
image of 
$\pi_{2-\bideg}\s_0\unit\iso H^{\star-2,\star}\to \pi_{1-\bideg}\f_1\unit$. 
The short exact sequence
\[ 0\to \f_1\pi_{1-\bideg}\unit \to \pi_{1-\bideg}\unit \to \pi_{1-\bideg}\s_0\unit = H^{\star-1,\star} \to 0 \]
then describes 
the difference between the first two $\KMW$-modules. In particular,
the canonical map $\pi_{1+w,w}\f_1\unit\to \pi_{1+w,w}\unit$ is an isomorphism
for $w>-2$. For the applications to motivic Moore spectra, an important
weight is $w=0$, where a short exact sequence
\begin{equation}\label{eq:pi10unit}
0\to \KMil_2/24 \to \pi_{1,0}\unit \to \ZZ/2\directsum \KMil_1/2 \to 0
\end{equation}
results. One result of Theorem~\ref{thm:pi1f1unit} is that 
the sequence~(\ref{eq:pi10unit}) splits as a sequence of abelian groups,
but not as a sequence of $\KMW_0$-modules. In particular,
$24\cdot \pi_{1,0}\unit =0$, in contrast with 
$24 \cdot \pi_{-1,-2}\unit_{\QQ} \neq 0$, which is similar to
$24 \cdot \pi_{-1,-2}\KGL_{\QQ} \iso 24\cdot \ZZ/48 \neq 0$.

\begin{lemma}\label{lem:action-forms-pi1}
  The action of $\GW$ on $\pi_{1+\bideg}\unit$ is determined
  by the following equations for $u\in F^\times$.
  \begin{align*}
    \langle u\rangle \cdot \nu & = \nu\in \pi_{3,2}\unit \\
    \langle u \rangle \cdot \eta\eta_{\Top} & = 
               \eta\eta_{\Top} + [u^{12}]\nu \in \pi_{2,1}\unit \\
    \langle u\rangle \cdot \eta_{\Top} & = 
           \eta_{\Top} + [u] \eta\eta_{\Top}\in \pi_{1,0}\unit
  \end{align*}
  In particular, $\hyper \cdot \eta_{\Top} = \rho\eta\eta_{\Top}$,
  $\hyper\cdot \eta\eta_{\Top}=0$, 
  and $\hyper\cdot \nu = 2\nu$. If $n$ is even, then
  $n\hyper$ acts on $\pi_{1+\bideg}\f_{1}\unit$ as multiplication by $2n$.
\end{lemma}

\begin{proof}
  This follows from Theorem~\ref{thm:pi1f1unit}, once
  the identification $\pi_{1+(w)}\f_1\unit = \pi_{1+(w)}\unit$
  for $w>-2$ is deduced from vanishing results for
  integral motivic cohomology of fields in weight $<2$.
\end{proof}

The situation for $\pi_{2+\bideg}\unit$ is a bit more delicate.
Nevertheless, the following can be read off from the
slice spectral sequence.

\begin{theorem}\label{thm:twoline}
  The element $\nu^2\colon \Sigma^{6,4}\unit\to \unit$
  induces an inclusion 
  \[ \KMW_{4+\star}\{\nu^2\}/(\eta\nu^2,2\nu^2)\to
  \pi_{2-\bideg}\unit\] 
  of $\KMW$-modules, which is an isomorphism for
  all $\star<-1$. In particular, 
  for all $w>4$, the group $\pi_{2+w,w}\unit \iso 0$.
\end{theorem}

\begin{proof}
  Since $\eta\nu=0$ \cite{dugger-isaksen.hopf}, also
  $\eta\nu^2=0$. Moreover, the $\varepsilon$-graded
  commutativity of $\pi_{\ast+\bideg}\unit$ implies that
  $\nu^2=-\nu^2$, whence a map 
  $\KMW_{4+\star}\{\nu^2\}/(\eta\nu^2,2\nu^2)\to
  \pi_{2-\bideg}\unit$ of $\KMW$-modules exists.
  The slice spectral sequence for $\pi_{2+\bideg}\unit$
  shows its injectivity, as well as the isomorphism
  statement, using results from \cite{rso.oneline}
  and \cite[Theorem 8.3]{roendigs.etainv}; details are
  to be given in~\cite{rso.twoline}.
\end{proof}

\section{Elementary properties of motivic Moore spectra}
\label{sec:elem-prop-motiv}

\begin{definition}
\label{defn:motivic-moore-spectrum}
Let $s\geq w \in \ZZ$, and let $\alpha\colon \Sigma^{s,w}\unit \to \unit$ be
an endomorphism. Choosing a homotopy cofiber sequence
\[ \Sigma^{s,w}\unit\xrightarrow{\alpha}\unit\xrightarrow{c}\moore{\alpha}\xrightarrow{d}\Sigma^{s+1,w}\unit \]
defines the motivic Moore spectrum $\moore{\alpha}$. 
\end{definition}

By definition, a motivic Moore spectrum for $a$ is unique up to equivalence.
The restriction $s\geq w$ is reasonable at least over a field by
Morel's connectivity theorem \cite{morel.connectivity}.

\begin{proposition}
\label{prop:homotopy-groups-moore}
Let $\alpha\colon \Sigma^{s,w}\unit\to\unit$ be an endomorphism. The canonical
maps induce a short exact sequence
\[ 0 \to \pi_{n+\bideg}\unit/\alpha\pi_{n-s+w+(\star-w)}\unit 
\xrightarrow{c_\ast} \pi_{n+\bideg}\moore{\alpha}
\xrightarrow{d_\ast} {}_{\alpha}\pi_{n-1-s+w+(\star-w)}\unit\to 0\]
of $\KMW$-modules, with target the submodule of elements annihilated
by $\alpha$.
\end{proposition}

\begin{proof}
  This follows from the homotopy cofiber sequence
  \[ \Sigma^{s,w}\unit\xrightarrow{\alpha}\unit\xrightarrow{c}\moore{\alpha}\xrightarrow{d}\Sigma^{s+1,w}\unit \]
  defining the motivic Moore spectrum.
\end{proof}

In particular, Morel's connectivity theorem implies 
with Proposition~\ref{prop:homotopy-groups-moore} that
$\pi_{0+\bideg}\moore{\alpha}\iso \KMW/\alpha\KMW$ if $\alpha\in \KMW$ (that is, if $s=w$).

\section{Toda brackets}
\label{sec:toda-brackets}

Consider three composable maps
\[ \D \xrightarrow{\gamma} \E \xrightarrow{\beta}\F \xrightarrow{\alpha} \GG \]
of motivic spectra such that $\beta\gamma = 0 = \alpha\beta$.
The {\em Toda bracket\/} 
\[ \langle \alpha,\beta,\gamma\rangle \quad \quad (\Mod \alpha\circ [\Sigma^{1,0}\D,\F]+[\Sigma^{1,0}\E,\GG] \circ \Sigma^{1,0}\gamma) \]
is the coset of the displayed subgroup of $[\Sigma^{1,0}\D,\GG]$ given
by those compositions $\Sigma^{1,0}\D\to \C(\beta) \to \GG$ such that
the obvious diagrams commute. The base scheme or field may be indicated
by a subscript. The most relevant case is where
all the motivic spectra involved are appropriate suspensions of
the motivic sphere spectrum $\unit$. See the classical 
source \cite{toda.composition}, as well as \cite{dugger-isaksen.real} and
\cite{guillou-isaksen.real} for interesting Toda brackets in
the motivic stable homotopy category. We consider
a few examples and make no claim to originality. 
Angled brackets are used
to denote both Toda brackets and quadratic forms (as in the paragraph after
Theorem~\ref{thm:zeroline}). 
The reader should be aware of this possible confusion, but
context will always make the meaning unambiguous.

\begin{proposition}\label{prop:toda-brackets}
The following equalities of subsets of $\pi_{\ast+\bideg}\unit$ hold.
\begin{align}\label{align:toda}
  \langle \hyper,\eta,\hyper \rangle & = 
  \bigl\lbrace \eta\eta_\Top+[\phi^2]\nu \colon \phi^2\in 
  2\KMil_{1}\{\nu\}/24\bigr\rbrace \\\label{align:toda-2}
  \langle \eta,\hyper,\eta\rangle & = \{6\nu,-6\nu\}  \\
  \label{align:toda-3}
  \langle \eta,\nu,\eta \rangle & = \{\nu^2 \}  \\
  \label{align:toda-4}
  \langle 2+11\hyper, \nu, 2+11\hyper\rangle_{\RR} & = \{\rho^2\nu^2\} 
\end{align}
\end{proposition}

\begin{proof}
  The indeterminacy in~(\ref{align:toda}) 
  is the subgroup $\hyper \circ \pi_{2,1}\unit = 2\KMil_{1}\{\nu\}/24$, using
  the isomorphism
  $\pi_{2,1}\unit\iso \KMil_0\{\eta\eta_{\Top}\}/2\oplus \KMil_{1}\{\nu\}/24$
  from Theorem~\ref{thm:pi1f1unit} and the equalities $[-1]\eta(\eta\eta_{\Top})=[-1]\eta^2\eta_\Top= [-1]12\nu=0=[-1]\eta\nu$ given in Lemma~\ref{lem:action-forms-pi1}. That the
  Toda bracket contains $\eta\eta_{\Top}$ follows from appropriate realization
  functors and the equality $\langle 2,\eta_{\Top},2\rangle_{\Top} = \{\eta_{\Top}\}$ in classical stable homotopy groups \cite{toda.composition}.
  
  The indeterminacy in~(\ref{align:toda-2})
  is $\eta \circ \pi_{2,1}=\{0,12\nu\}$, since 
  Theorem~\ref{thm:pi1f1unit} implies 
  $\pi_{2,1}\unit\iso \KMil_0\{\eta\eta_{\Top}\}/2\oplus \KMil_{1}\{\nu\}/24$,
  and $\eta\eta\eta_{\Top}=12\nu$, whereas $\eta\nu=0$.
  Complex
  realization maps $\langle \eta,\hyper,\eta\rangle$ to
  $\langle \eta_\Top,2,\eta_{\Top} \rangle_\Top = \{6\nu_{\Top},-6\nu_{\Top}\}$.
  The relevant group $\pi_{3,2}\unit$ does not depend on the base
  field (at least after inverting the exponential characteristic),
  giving the result.
  
  The indeterminacy in~(\ref{align:toda-3}) is the subgroup
  $\eta\circ \pi_{4,2} = \{0\}$, since $\pi_{4,2}\unit \iso \KMil_2\{\nu^2\}/2$
  by Theorem~\ref{thm:twoline}, and $\eta\nu^2=0$. The unique element
  in $\langle \eta, \nu,\eta\rangle$ has to be the unique nontrivial
  one in $\pi_{6,4}\unit = \ZZ/2$ by appropriate realization.

  Similar to the previous case, the indeterminacy in~(\ref{align:toda-4})
  is $\{0\}$ by Theorem~\ref{thm:twoline}, since $(2+11\hyper)\nu^2=0$.
  Real realization sends 
  $\langle 2+11\hyper, \nu, 2+11\hyper\rangle_{\RR}$ to
  $\langle 2, \eta_{\Top},2\rangle_{\Top}=\{\eta_{\Top}^2\}$, whence
  the previous Toda bracket has to contain the unique nonzero element
  in $\pi_{4,2}\unit_{\RR}$. 
\end{proof}

The following statement, which could be formulated in greater
generality, indicates the relevance of Toda brackets for 
the structure of motivic Moore spectra. The notation is as in
Proposition~\ref{prop:homotopy-groups-moore}.
 
\begin{proposition}\label{prop:toda}
  Suppose $\alpha\in \pi_{s,w}\unit$,
  $\beta\in \pi_{t,x}\unit$, and $\gamma\in \pi_{u,y}\unit$
  are elements with $\alpha\beta=0=\beta\gamma$. The Toda bracket
  $\langle \alpha,\beta,\gamma\rangle$ coincides with the set
  of elements $\delta\in \pi_{s+t+u+1,w+x+y}\unit$ for which there
  exists an element $\tilde{\beta}\in \pi_{s+t+1,w+x}\moore{\alpha}$ with 
  $d_\ast(\tilde{\beta}) = \beta$ and $c_\ast(\delta+\alpha\cdot \pi_{t+u+1,x+y}\unit)=\tilde{\beta}\cdot \gamma$.
\end{proposition} 

\begin{proof}
  The proof consists of comparing various diagrams in the motivic
  stable homotopy category and is left as an exercise to the reader.
\end{proof}

\begin{example}\label{ex:toda-eta-nu-eta}
  Consider $\alpha=\gamma=\eta\in \pi_{1,1}\unit$. Then
  $\pi_{5,3}\moore{\eta}$ contains an element $\tilde{\nu}$ with
  $\tilde{\nu}\cdot \eta\neq 0$. In fact, the Toda bracket
  $\langle \eta,\nu,\eta\rangle =\{\nu^2\}$ given in~(\ref{align:toda-3})
  contains a single
  nonzero element, and $c_\ast\colon \pi_{6,4}\unit \to \pi_{6,4}\moore{\eta}$
  is injective, since multiplication with $\eta$ is the
  zero map $\pi_{5,3}\unit = h^{1,1}\{\nu^2\}\to h^{0,0}\{\nu^2\}=\pi_{6,4}\unit$
  by Theorem~\ref{thm:twoline}.
\end{example}

\begin{example}\label{ex:toda-reals}
  Let $\alpha=\gamma=2+11\hyper \in \pi_{0,0}\unit_{\RR}$ 
  (or any formally real field). Then $\pi_{4,2}\moore{2+11\hyper}$
  contains an element $\tilde{\nu}$ with 
  $\tilde{\nu}\cdot (2+11\hyper)\neq 0$. Indeed, the Toda bracket
  given in~(\ref{align:toda-4})
  contains a nonzero element (uniquely for $\RR$), and
  $c_\ast\colon \pi_{4,2}\unit_{\RR}\to \pi_{4,2}\moore{2+11\hyper}$
  is injective, since multiplication with $2+11\hyper$ induces
  the zero map $\pi_{4,2}\unit_{\RR}=h^{2,2}\{\nu^2\}\to \pi_{4,2}\unit_{\RR}$ by
  Theorem~\ref{thm:twoline}.
\end{example}

\section{Multiplications}
\label{sec:multiplications}

The ``constant presheaf'' functor defines a strict symmetric monoidal
triangulated functor 
$\const\colon \SH\to \SH(S)$ for any base scheme $S$.
The motivic Moore spectra $\moore{n}$ for $n\in \NN$ are in its image.
In particular, multiplications or pairings on these motivic 
Moore spectra can be transferred from the corresponding
topological ones \cite{oka.mult-moore}.

\begin{lemma}\label{lem:mult-order-identity}
  Let $\alpha\colon \Sigma^{s,w}\unit\to \unit$ be an endomorphism and 
  $c\colon \unit\to \moore{\alpha}$ the map to the homotopy cofiber.
  There exists a left unital pairing 
  \[ \moore{\alpha}\smash \moore{\alpha}\to\moore{\alpha}\]
  if and only if the identity on $\moore{\alpha}$ is annihilated by $\alpha$.
\end{lemma}

\begin{proof}
  Consider the following diagram:
  \[ \xymatrix{ \unit \smash \unit = \unit \ar[d]_\id \ar[r]^c & 
    \moore{\alpha}=\unit\smash \moore{\alpha} \ar[r]^{c\smash \moore{\alpha}} \ar[d]^\id& \moore{\alpha}\smash \moore{\alpha} \\
    \unit \ar[r]^c &\moore{\alpha}\ar[r]^\id & \moore{\alpha}}\]
  There exists a left unital pairing 
  \[ \moore{\alpha}\smash \moore{\alpha}\to\moore{\alpha}\]
  if and only if the map
  $c\smash \moore{\alpha}$ admits a retraction. After smashing with $\moore{\alpha}$, 
  the homotopy cofiber sequence
  \[ \Sigma^{s,w}\unit\xrightarrow{\alpha}\unit \xrightarrow{c} \moore{\alpha} \xrightarrow{d}\Sigma^{s+1,w} \unit\]
  induces a long exact sequence
  \[ \dotsm \xleftarrow{d^\ast} [\Sigma^{s,w}\moore{\alpha},\moore{\alpha}] 
  \xleftarrow{\alpha} [\moore{\alpha},\moore{\alpha}] 
  \xleftarrow{c^\ast} [\moore{\alpha}\smash \moore{\alpha},\moore{\alpha}] 
  \xleftarrow{d^\ast} [\Sigma^{s+1,w}\moore{\alpha},\moore{\alpha}] 
  \xleftarrow{\alpha} \dotsm \]
  which shows the desired statement.
\end{proof}

If a left unital pairing 
$\mu\colon \moore{\alpha}\smash \moore{\alpha}\to \moore{\alpha}$
exists, then $\mu\circ (\moore{\alpha}\smash c)\circ c =c$.
The short exact sequence
\[ 0\to \pi_{s+1,w}\moore{\alpha}
\xrightarrow{d^\ast} [\moore{\alpha},\moore{\alpha}]
\xrightarrow{c^\ast} \pi_{0,0} \moore{\alpha} \to 0\]
then supplies an element $\psi\in \pi_{s+1,w}\moore{\alpha}$ with
$\psi\circ d = \mu\circ (\moore{\alpha}\smash c) -\id_{\moore{\alpha}}$.
The short exact sequence
\[ 0\to \pi_{2s+2,2w}\moore{\alpha}
\xrightarrow{d^\ast} [\Sigma^{s+1,w}\moore{\alpha},\moore{\alpha}]
\xrightarrow{c^\ast} \pi_{s+1,w} \moore{\alpha} \to 0\]
provides $\theta\colon \Sigma^{s,w}\moore{\alpha}\to\moore{\alpha}$
with $\theta \circ \Sigma^{s+1,w}d=\psi$. It follows that the 
left unital pairing $\mu-\theta\circ (d\smash \moore{\alpha})$ 
is also right unital, because
\[ \bigl(\mu-\theta\circ (d\smash \moore{\alpha})\bigr)\circ 
(\moore{\alpha}\smash c) = 
\mu\circ (\moore{\alpha}\smash c) - \theta\circ (d\smash c) = 
\mu\circ (\moore{\alpha}\smash c) -\theta\circ \Sigma^{s+1,w}c\circ d 
=\id_{\moore{\alpha}}.\] Hence if a left unital pairing on
$\moore{\alpha}$ exists, a unital pairing exists as well. 
In the following, ``multiplication'' stands for ``unital pairing''.

\begin{lemma}\label{lem:order-id-moore}
  Let $\alpha\colon \Sigma^{s,w}\unit\to \unit$ be an endomorphism.  
  Then $\alpha^{2}\cdot \id_{\moore{\alpha}}=0$. 
\end{lemma}

\begin{proof}
  The homotopy cofiber sequence defining $\moore{\alpha}$ 
  induces a short exact sequence
  \begin{equation}\label{eq:id-moore} 
    0\to [ \Sigma^{s+1,w}\unit,\moore{\alpha}]/\alpha
    \xrightarrow{d^\ast} [\moore{\alpha},\moore{\alpha}]
    \xrightarrow{c^\ast} {}_{\alpha}[\unit,\moore{\alpha}] =[\unit,\moore{\alpha}] = [\unit,\unit]/\alpha \to 0 
  \end{equation}
  for which Proposition~\ref{prop:homotopy-groups-moore} implies
  the identification of the outer terms.
  In particular, $c^\ast(\id_{\moore{\alpha}})$ is annihilated by $\alpha$. Hence
  $\alpha\cdot \id_{\moore{\alpha}}=d^\ast(x)$ for some 
  $x\in [\Sigma^{s+1,w}\unit,\moore{\alpha}]/\alpha$. Thus $\alpha\cdot x=0$,
  which proves the equality 
  $\alpha^{2}\cdot \id_{\moore{\alpha}}=\alpha\cdot d^\ast(x)=d^\ast(\alpha\cdot x)=0$.
\end{proof}

As a consequence, there exists a pairing
\[ \moore{\alpha}\smash \moore{\alpha^{2}}\to \moore{\alpha} \]
which in the classical case of the topological sphere
spectrum and $\alpha=2$ was described by Oka in \cite{oka.mult-moore}.
For example, using $\hyper^2=2\hyper$, there results a pairing
\[ \moore{\hyper}\smash \moore{2\hyper}\to \moore{\hyper} \]
which realizes to Oka's pairing for subfields of the complex
numbers. The specific role that squares play will be clarified
by the following statement lifting a theorem of Brayton Gray 
\cite[Theorem 10]{gray.ring-aarhus}.

\begin{theorem}\label{thm:mult-moore-square}
  Let $\alpha\colon \Sigma^{s,w}\unit\to \unit$ be any endomorphism.
  Then $\moore{\alpha^2}$ admits a multiplication.
\end{theorem}

\begin{proof}
  Lemma~\ref{lem:mult-order-identity} implies it suffices
  to show that $\alpha^2\cdot \id_{\moore{\alpha^2}}=0$. 
  In order to prove this, consider the following
  general construction for $\alpha\colon \D\to \E$,
  modeled on \cite{toda.composition}, and observing
  that the $\varepsilon$-graded ring structure on
  $\pi_{\ast+\bideg}\unit$ may be equally defined via
  composition or via smash product.
  Let $R(\alpha)$ be the set of all maps
  $A\colon \Sigma^{1,0}\D\smash \D\to \E\smash \E$ such
  the diagram
  \[ \begin{tikzcd}
    \Sigma^{1,0}\D\smash \D\ar[r,"A"] &  \E\smash \E \ar[d,"\E\smash c"] \\
    \D\smash \moore{\alpha} \ar[r,"\alpha\smash \moore{\alpha}"] 
    \ar[u,"\D\smash d"] & \E\smash \moore{\alpha}
  \end{tikzcd} \]
  commutes. 
  If $R(\alpha)$ contains an element
  of the form $\beta\smash \alpha$, where $\beta\colon \Sigma^{1,0}\D\to \E$,
  then $\alpha\smash \moore{\alpha}=\alpha\cdot \id_{\moore{\alpha}}$ 
  is the zero map. Several natural commutative diagrams show that if
  $\alpha\colon \D\to \E$ and $\beta \colon \mathsf{F}\to \mathsf{G}$
  are two maps, then for every $A\in R(\alpha)$, the element
  \[ \Sigma^{1,0}\D\smash \mathsf{F}\smash \D\smash \mathsf{F}
  \xrightarrow{\mathrm{twist}} 
  \Sigma^{1,0} \D\smash \D\smash \mathsf{F}\smash \mathsf{F}
  \xrightarrow{A\smash \beta\smash\beta} 
  \E\smash \E\smash \mathsf{G}\smash \mathsf{G}
  \xrightarrow{\mathrm{twist}} 
  \E\smash \mathsf{G}\smash \E \smash \mathsf{G} \]
  lies in $R(\alpha\smash \beta)$. In particular, 
  for $\alpha\in \pi_{s,w}\unit$ and $\beta\in \pi_{t,x}\unit$,
  the inclusion 
  $(-1)^{(s-w)(t-x)}\varepsilon^{wx}R(\alpha)\beta^2\subset R(\alpha\smash \beta)$
  holds, using \cite[Equation (2.4)]{dugger-isaksen.hopf}. 
  In the special case $\alpha=\beta$, one obtains the inclusion
  $(-1)^{(s-w)^2}\varepsilon^{w^2}R(\alpha)\alpha^2\subset R(\alpha\smash \alpha)$.
  If $R(\alpha)$ is not empty, this implies that 
  $\alpha^2\smash \moore{\alpha^2}$ is zero. It remains to see that $R(\alpha)$
  is not empty for $\alpha\in \pi_{s,w}\unit$. Since $\alpha\smash \moore{\alpha}
  \circ \Sigma^{s,w}c=c\circ \Sigma^{s,w}\alpha=0$, there exists
  $\gamma\in \pi_{2s+1,2w}\moore{\alpha}$ with 
  $\gamma\circ \Sigma^{s,w}d=\alpha\smash \moore{\alpha}$. 
  If $d\circ \gamma=0$, then $R(\alpha)$ contains a lift of $\gamma$
  along $c$. If $d\circ \gamma\neq 0$, then 
  the equation 
  \[ d\circ \gamma \circ \Sigma^{s,w}d = d\circ \alpha\smash \moore{\alpha} =
  \Sigma^{s+1,w}\alpha\circ \Sigma^{s,w}d=0\]
  supplies $\phi\in \pi_{0,0}\unit$ with $\phi \circ \alpha=d\circ \gamma$. 
  Since $\alpha\circ\phi\circ \alpha = 0$, there exists 
  $\psi\in \pi_{s+1,w}\moore{\alpha}$ with $d\circ \psi \circ \alpha =\phi\circ \alpha$
  It follows that $(\gamma-\psi\circ \alpha)\circ d = \gamma\circ d = \alpha\smash
  \moore{\alpha}$ and $d\circ (\gamma - \psi\circ \alpha)=0$. A lift
  of $\gamma-\psi\circ \alpha\in \pi_{2s+1,2w}\moore{\alpha}$ along $c$ provides
  an element in $R(\alpha)$, which concludes the proof. 
\end{proof}

Now for some negative results regarding multiplications on 
motivic Moore spectra. Recall that $\pi_{0,0}\unit = GW$ identifies
with the Grothendieck-Witt ring of quadratic forms by 
Theorem~\ref{thm:zeroline}, which comes equipped with a dimension
ring homomorphism $\dim \colon \GW \to \ZZ$.

\begin{theorem}\label{thm:moore-forms-nomult}
  Let $\alpha\colon \unit \to \unit$ be an endomorphism
  with $\dim(\alpha)\equiv 2(4)$. Then
  $\moore{\alpha}$ does not admit a multiplication.
\end{theorem}

\begin{proof}
  This result follows by complex realization from \cite{oka.mult-moore}
  for subfields of the complex numbers. In any case, consider 
  motivic cohomology $h^{\ast,\ast}$
  with coefficients in $\FF_2$. Since multiplication
  with $\alpha$ induces multiplication with $\dim(\alpha)$ on 
  motivic cohomology,
  there results a split short exact sequence
  \[ 0 \to h^{\ast,\ast}\to h^{\ast,\ast}(\moore{\alpha}) \to h^{\ast +1,\ast} \to 0\]
  on motivic cohomology. Pick basis elements $x_0\in h^{0,0}(\moore{\alpha}),
  x_1\in h^{1,0}(\moore{\alpha})$, and note that $\Sq^1(x_0)=x_1$
  since $\dim(\alpha)\equiv 2(4)$. 
  The split short exact sequence 
  \[ 0 \to h^{\ast,\ast}(\moore{\alpha}) \to h^{\ast,\ast}(\moore{\alpha}\smash \moore{\alpha})
  \to h^{\ast+1,\ast}(\moore{\alpha}) \to 0\]
  induced by multiplication with $\alpha$ on $\moore{\alpha}$
  then supplies $h^{\ast,\ast}(\moore{\alpha}\smash \moore{\alpha})$
  with basis elements 
  $x_0\tensor y_0,x_1\tensor y_0,x_0\tensor y_1,x_1\tensor y_1$.
  Here $y_0\in h^{0,0}(\moore{\alpha})$ and 
  $y_1\in h^{1,0}(\moore{\alpha})$ are basis elements in 
  the other factor of the smash product
  $\moore{\alpha}\smash \moore{\alpha}$. A K\"unneth theorem 
  supplies the basis elements for the motivic cohomology
  of the smash product.
  The Cartan formula
  \[ \Sq^2(x_0\tensor y_0) = \Sq^2(x_0)\tensor y_0 + x_0\tensor \Sq^2(y_0)
  +\tau \Sq^1(x_0)\tensor \Sq^1(y_0) = \tau x_1\tensor y_1 \]
  from \cite[Proposition 9.7]{voevodsky.reduced} then shows that $\Sq^2$ acts nontrivially
  on $h^{\ast,\ast}(\moore{\alpha}\smash \moore{\alpha})$. 
  If $\alpha\cdot \id_{\moore{\alpha}}=0$,
  then $\moore{\alpha}\smash \moore{\alpha}=\moore{\alpha}\vee \Sigma^{1,0}\moore{\alpha}$,
  and $h^{\ast,\ast}(\moore{\alpha}\smash \moore{\alpha})$ splits accordingly as a 
  module over the motivic Steenrod algebra,
  implying $\Sq^2(x_0\tensor y_0)=0$. The result follows.
\end{proof}

\begin{lemma}\label{lem:mult-moore-eta}
  The motivic Moore spectrum $\moore{\eta^\ell}$ admits a multiplication
  if and only if $\ell>1$.
\end{lemma}

\begin{proof}
  Let us prove first that $\moore{\eta}$ does not admit a multiplication.
  Again an argument via complex realization works for subfields
  of the complex numbers. However, one may refer to 
  Example~\ref{ex:toda-eta-nu-eta} and Proposition~\ref{prop:toda},
  which imply the existence of an element 
  $\tilde{\nu}\in \pi_{3,2}\moore{\eta}$ such that 
  $\tilde{\nu}\cdot \eta\neq 0$. In particular, 
  $\eta\cdot \id_{\moore{\eta}}\neq 0$. 

  Consider now $\ell >1$. The element $\eta^\ell\cdot \id_{\moore{\eta^\ell}}$
  lies inside the group which
  sits in the middle of the short exact sequence
  \[ 0\to [ \Sigma^{2\ell+1,2\ell}\unit,\moore{\eta^\ell}]/\eta^\ell
  \xrightarrow{d^\ast} [\Sigma^{\ell,\ell}\moore{\eta^\ell},\moore{\eta^\ell}]
  \xrightarrow{c^\ast} {}_{\eta^\ell}[\Sigma^{\ell,\ell}\unit,\moore{\eta^\ell}] = 
  [\Sigma^{\ell,\ell}\unit,\unit]/\eta^\ell \to 0 \]
  whose final term is zero, since 
  $\eta^\ell\colon \pi_{0,0}\unit\to \pi_{\ell,\ell}\unit$ is surjective.
  The initial term is computed by 
  Proposition~\ref{prop:homotopy-groups-moore}. More
  precisely, the group $\pi_{2\ell+1,2\ell}\moore{\eta^\ell}$ sits inside
  the short exact sequence
  \[ 0 \to \pi_{2\ell+1,2\ell}\unit/\eta^\ell\pi_{\ell+1,\ell}\unit 
  \xrightarrow{c_\ast}
  \pi_{2\ell+1,2\ell}\moore{\eta^\ell}
  \xrightarrow{d_\ast} {}_{\eta^{\ell}}\pi_{\ell,\ell}\unit\to 0 \]
  where ${}_{\eta^{\ell}}\pi_{\ell,\ell}\unit =0$ since
  $\eta^\ell \colon \pi_{\ell,\ell}\unit \to \pi_{2\ell,2\ell}\unit$
  is an isomorphism for $\ell \geq 1$. The group
  $\pi_{2\ell+1,2\ell}\unit$ vanishes by Theorem~\ref{thm:oneline} for $\ell>1$,
  whence so does $\pi_{2\ell+1,2\ell}\moore{\eta^\ell}$, and
  thus also $[\Sigma^{\ell,\ell}\moore{\eta^\ell},\moore{\eta^\ell}]$.
\end{proof}

The existence of a multiplication may depend on the base field
in general. Example~\ref{ex:toda-reals} and
Lemma~\ref{lem:mult-order-identity} show that the motivic Moore spectrum
$\moore{2+11\hyper}$ admits no multiplication over a formally real
field $F$. Base change to $F(\sqrt{-1})$ produces the motivic
Moore spectrum $\moore{2+11\hyper} = \moore{24}$, which admits
a multiplication by lifting from topology and quoting \cite{oka.mult-moore}.
Plenty of similar examples may be constructed.
A systematic study on these matters remains a project for the future.
Another project for the future is to enumerate possible multiplications,
as well as investigate their qualitative properties like associativity
and commutativity. Already the enumeration can be challenging. For
example, the group $[\moore{\hyper^2}\smash\moore{\hyper^2},\moore{\hyper^2}]$
receives a nontrivial map from $[\Sigma^{1,0}\moore{\hyper^2},\moore{\hyper^2}]$,
which is injective on the contribution
from the $\hyper^2$-torsion in $\pi_{0,0}\unit$ 
(see Proposition~\ref{prop:homotopy-groups-moore}). The latter coincides
with the fundamental ideal in the Grothendieck-Witt ring, 
as the following statement implies.

\begin{lemma}
\label{lem:hyper-torsion}
  Let $\alpha\in \GW(F)$ be an element with $\dim(\alpha)\neq 0$.
  Then 
  \[ {}_{\alpha{\hyper} }\GW(F) ={}_{\dim(\alpha){\hyper} }\GW(F) = {}_{{\hyper} }\GW(F)= \I(F).\]
\end{lemma}

\begin{proof}
  Recall that $\hyper = 1-\epsilon=\langle 1,-1\rangle$ is a form of dimension
  two with the property that the subgroup of $\GW(F)$
  generated by $\hyper$ coincides
  with the ideal of $\GW(F)$ generated by $\hyper$. The latter follows
  from the similarity
  \[ \langle u \rangle \cdot \langle 1,-1\rangle = \langle u,-u\rangle
  \sim \langle 1,-1\rangle\]
  of quadratic forms, where $u\in F^\times$. Hence for every 
  $\alpha\in \GW(F)$, the
  equation $\alpha\cdot \hyper = \dim(\alpha)\cdot \hyper$ follows, giving
  the first equation. Also the other equations follow for $\dim(\alpha)\neq 0$, 
  since $\beta\in \GW(F)$ then satisfies $\beta\cdot (\alpha\cdot \hyper)=0$
  if and only if $\dim(\beta)=0$.
  Here recall the short exact sequence
  \[ 0 \to \I(F) \to \GW(F) \xrightarrow{\dim} \ZZ \to 0 \]
  defining the fundamental ideal in $\GW(F)$.
\end{proof}

Prompted by a recent request, this section concludes with a specific example.
Let $n>0$ be a natural number, and let
$n_\varepsilon=\sum\limits_{k=0}^{n-1}\langle (-1)^k\rangle = 
\langle 1, -1,1\dotsc,\pm 1\rangle \in \GW$, a quadratic form
of dimension $n$. 
It turns out that the
motivic Moore spectrum $\moore{n_\varepsilon}$ 
admits a multiplication precisely if the topological
Moore spectrum $\sphere/n$ does. Before proving this,
set $e\colon \moore{\alpha}\xrightarrow{d}\Sigma^{s+1,w}\unit
\xrightarrow{\Sigma^{s+1,w}c}\Sigma^{s+1,w}\moore{\alpha}$ for
$\alpha\in \pi_{s,w}\unit$. Following \cite{oka.mult-moore}, a multiplication 
$\mu\colon \moore{\alpha}\smash \moore{\alpha}\to \moore{\alpha}$
is called \textit{regular} if the equality 
$e\circ \mu = (\Sigma^{s+1,w}\mu)\circ (e\smash \moore{\alpha} + \moore{\alpha}\smash e) 
\colon \moore{\alpha}\smash \moore{\alpha}\to \Sigma^{s+1,w}\moore{\alpha}$
holds.

\begin{lemma}\label{lem:nepsilon}
There exists a multiplication on $\moore{n_\varepsilon}$
if and only if $n\not\equiv 2 (4)$. 
\end{lemma}

\begin{proof}
  Theorem~\ref{thm:moore-forms-nomult} says that $\moore{n_\varepsilon}$
  does not admit a multiplication if $n\equiv 2(4)$. Suppose
  first that $n=2m+1$ is an odd natural number. Any element $\alpha\in \GW$
  with $n_\varepsilon\alpha = 0$ then satisfies $\dim(\alpha)=0$, and
  hence lies in the fundamental ideal. Since
  $n_\varepsilon= 1+m\hyper$ and the element $\hyper\alpha$ is hyperbolic of 
  dimension zero, one concludes $\alpha=0$. 
  Proposition~\ref{prop:homotopy-groups-moore} then provides an isomorphism 
  $\pi_{1,0}\moore{n_\varepsilon}\iso \pi_{1,0}\unit/n_\varepsilon\pi_{1,0}\unit$.
  Lemma~\ref{lem:action-forms-pi1} implies that the map
  $\pi_{1,0}\unit\xrightarrow{n_\varepsilon} \pi_{1,0}\unit$ is surjective
  if $n\not\equiv 0(3)$. Hence if $n$ is odd and not divisible by three,
  $\pi_{1,0}\moore{n_\varepsilon}=0$. The short exact 
  sequence~(\ref{eq:id-moore}) then reduces to
  $[\moore{n_\varepsilon},\moore{n_\varepsilon}]\iso \pi_{0,0}\moore{n_\varepsilon}$.
  In particular, the identity on $\moore{n_\varepsilon}$ is annihilated by
  $n_\varepsilon$, which provides the existence of a multiplication
  by Lemma~\ref{lem:mult-order-identity}. If $n$ is odd and divisible
  by three, $\pi_{1,0}\moore{n_\varepsilon}\iso \mathbf{K}^\Mil_2/3$. 
  In order to conclude for such odd numbers as well, it suffices to
  prove that the short exact sequence
  \begin{equation}\label{eq:2} 
    0 \to \pi_{1,0}\moore{n_\varepsilon}\iso \mathbf{K}^\Mil_2/3
    \to [\moore{n_\varepsilon},\moore{n_\varepsilon}] 
    \to \pi_{0,0}\moore{n_\varepsilon} \to 0 
  \end{equation}
  splits. Because $n_\varepsilon=1+m\cdot \hyper$ acts as the identity on the
  Witt group, $\pi_{0,0}\moore{n_\varepsilon}\iso \ZZ/n$ for any field,
  generated by $\hyper$. Hence if the sequence~(\ref{eq:2}) splits
  over prime fields, it does so over any field, reducing the task
  to $F=\mathbb{Q}$. In this case, \cite[Theorem 11.6]{milnor.k} 
  supplies an isomorphism 
  $\mathbf{K}^\Mil_2/3(\QQ)
  \iso \bigoplus\limits_{p \mathrm{\ prime}, p\equiv 1(3)} \ZZ/3$ whose contribution
  at the prime $p\equiv 1(3)$ is induced by the localization sequence
  \[ \Spec(\FF_p) \hookrightarrow \Spec(\ZZ_{(p)}) \hookleftarrow \Spec(\QQ).\]
  Since both $\hyper$ and the element $n_\varepsilon$ are defined over 
  $\Spec(\ZZ)$, it follows that the sequence~(\ref{eq:2}) splits. 
  
  Suppose now that $n=2^rm$ for some odd natural number $m$, with $r\geq 2$. 
  Then $n_\varepsilon= m\hyper^r$. By Theorem~\ref{thm:mult-moore-square},
  $\moore{\hyper^2}$ admits a multiplication $\mu$. If $\mu$ is not regular,
  the difference 
  $x=e\circ \mu - \mu\circ (e\smash \moore{\hyper^2} + \moore{\hyper^2}\smash e)$ 
  lifts to produce a unique element $y\in \pi_{1,0}\moore{\hyper^2}$
  such that $y\circ (d\smash d) = x$. The latter equation shows that
  $x$ is an element of order two, because $d\smash d = - d\smash d$, 
  hence so is $y$. The vanishing $0=d\circ y\in \pi_{0,0}\moore{\hyper^2}$
  implies that there exists $z\in \pi_{1,0}\unit$ with $c\circ z=y$. The image
  of $\pi_{1,0}\unit$ under composition with $c$ is isomorphic, by 
  Theorem~\ref{thm:pi1f1unit}, to 
  $\mathbf{K}^\Mil_2/4 \directsum
  \mathbf{k}^\Mil_1\directsum \mathbf{k}^\Mil_0$. The map 
  $\lambda\colon \moore{\hyper^2}\to 
  \moore{\hyper^4}$ induced by multiplication with $\hyper^2$ on the bottom cell
  does not necessarily induce the zero homomorphism on $\pi_{1,0}$.
  More specifically, it induces multiplication by 4 as a homomorphism
  \[ \mathbf{K}^\Mil_2/4 \directsum
  \mathbf{k}^\Mil_1\directsum \mathbf{k}^\Mil_0 \to \mathbf{K}^\Mil_2/8 \directsum
  \mathbf{k}^\Mil_1\directsum \mathbf{k}^\Mil_0 \] 
  and in particular sends the 2-torsion element $y$ to zero. The homotopy cofiber 
  sequence 
  \[ \moore{\hyper^2}\xrightarrow{\lambda }\moore{\hyper^2}\xrightarrow{\rho}
  \moore{\hyper^2}\xrightarrow{e} \Sigma^{1,0}\moore{\hyper^2} \]
  implies the existence of $w\in \pi_{2,0}\moore{\hyper^2}$ with $e\circ w= y$.
  Then $\mu - w\circ (d\smash d)$ is a regular multiplication on $\moore{\hyper^2}$. Hence a regular multiplication exists on $\moore{\hyper^2}$.

  Inductively, one may lift
  a regular multiplication on $\moore{\hyper^{\ell}}$ to 
  a regular multiplication on $\moore{\hyper^{\ell+1}}$ via 
  the homotopy cofiber sequence
  \[ \moore{\hyper}\xrightarrow{\lambda}\moore{\hyper^{\ell+1}} \xrightarrow{\rho}
  \moore{\hyper^{\ell}}\xrightarrow{\delta}  \moore{\hyper}. \]
  Here $\lambda$ and $\rho$ are induced by suitable multiplications with
  $\hyper^{\ell}$ on the bottom cell and $\hyper$ on the top cell, respectively, and 
  $\delta$ is the composition 
  $\moore{\hyper^{\ell}}\xrightarrow{d} \Sigma^{1,0}\unit\xrightarrow{\Sigma^{1,0}c}
  \moore{\hyper}$. 
  The homotopy cofiber sequence
  \[ \moore{\hyper}\xrightarrow{\lambda}\moore{m} \xrightarrow{\rho}
  \moore{m\hyper^{r}}\xrightarrow{\delta}  \moore{\hyper^r} \]
  allows to lift a regular multiplication $\mu$ on $\moore{\hyper^r}$ to
  a regular multiplication on $\moore{n_{\varepsilon}}$. Over quadratically
  closed fields, $\delta=0$ (as in topology, 
  see \cite[Lemma 5]{oka.mult-moore}), but $\delta\neq 0$ for
  formally real fields. Nevertheless the equality
  $\delta \circ \mu \circ (\rho\smash \rho) =0$ always holds. This
  computation follows from the fact that $\mu$ is regular and
  two other inputs. One input is the
  result from topology that $m$ is zero on $\sphere/m$, and hence
  also on its image $\moore{m}$ in the motivic stable homotopy category.
  In order to apply the regularity, one uses that the composition 
  $\moore{\hyper^r}\xrightarrow{d}\Sigma^{1,0}\unit
  \xrightarrow{\Sigma^{1,0}c}\Sigma^{1,0}\moore{m}$ factors as 
  $\moore{\hyper^r}\xrightarrow{d}\Sigma^{1,0}\unit
  \xrightarrow{\Sigma^{1,0}c}\Sigma^{1,0}\moore{\hyper^r}
  \xrightarrow{\Sigma^{1,0}\phi} \Sigma^{1,0}\moore{m}$ for some
  $\phi\colon \moore{\hyper^r}\to \moore{m}$. This provides
  a regular multiplication on $\moore{n_{\varepsilon}}$
  for $n\not\equiv 2(4)$, which concludes the proof.
\end{proof}


\section{Slices of motivic Moore spectra}
\label{sec:slices-moore-spectra}

Instead of considering slices for general motivic Moore spectra,
the focus here -- motivated by arguments regarding the
vanishing of higher slice differentials in \cite[Section 4]{rso.oneline} -- 
is on $\moore{n\hyper}$, 
where $0<n\in \NN$. As explained in the proof of Lemma~\ref{lem:hyper-torsion},
it suffices to consider motivic Moore spectra with respect
to $n\hyper$ instead of 
$\alpha\hyper$ for $\alpha\in \GW(F)$ of $\dim(\alpha)\neq 0$.
Since $\hyper = 0 \colon \KW \to \KW$ as an endomorphism
on the motivic spectrum representing higher Witt groups, the 
motivic spectrum $\KW_{n\hyper}$
splits as $\KW\vee \Sigma^{1,0}\KW$. 
The same holds for its (effective or connective) covers,
and also for the corresponding slices. In particular, the first
slice differential for $\KW_{n\hyper}$ splits. For reference
purposes, the explicit form is as follows.

\begin{theorem}
\label{thm:ckwnhyper-diff}
Let $0<n\in \NN$. The restriction of the slice $\dd^{1}$-differential to the summand $\Sigma^{q + j,q}\MZ/2$ of $\s_{q}(\KW_{n\hyper})$ is given by 
\begin{equation*}
\dd^{1}(\KW_{n\hyper})(q,j)
= 
\begin{cases}
(\Sq^{3}\Sq^{1},0,\Sq^{2}) & j\equiv 0,1\bmod 4 \\
(\Sq^{3}\Sq^{1},0,\Sq^{2}+\rho\Sq^{1},0,\tau) & j\equiv 2,3\bmod 4 \\
\end{cases}
\label{equation:ckwnhyper-diff1}
\end{equation*}
Here the $i$th component of the map $\dd^{1}(\KW_{n\hyper})(q,j)$ of motivic
spectra is a map $\Sigma^{q+j,q}\MZ/2\to\Sigma^{q+j+i,q+1}\MZ/2$. 
\end{theorem}

\begin{proof}
  This follows from the determination of the first slice differential
  for $\KW$ from \cite{roendigs-oestvaer.hermitian}, and the aforementioned
  splitting $\KW_{n\hyper}\simeq \KW \vee \Sigma^{1,0}\KW$.
\end{proof}

Here and in the following, the notation regarding the motivic Steenrod
algebra is standard; for example, $\rho$ is the class of $-1$ in $h^{1,1}$, and
$\Q_1=\Sq^2\Sq^1+\Sq^1\Sq^2$. 
The slice computation is slightly more complicated for the motivic spectrum
$\KQ$ representing hermitian $K$-theory. For comparison purposes, the case of
$\kq$,
the very effective cover of $\KQ$ \cite{aro.kq}, \cite{bachmann.generalized},
is more convenient.
Set $\kq_{n\hyper}:=\kq\smash \moore{n\hyper}$,
and similarly $\ckw_{n\hyper}:=\ckw\smash \moore{n\hyper}=
\ckw \vee \Sigma^{1,0}\ckw$. Here $\ckw := \KW_{\geq 0} \iso \kq[\eta^{-1}]$
is the connective cover of $\KW$, not its very effective cover
(which might deserve the notation $\kw$). The element
$\eta$ acts still invertibly on $\ckw$. 

The final set of notation concerns long exact
sequences of motivic cohomology groups induced
by change of coefficients in cyclic groups.
Given natural numbers $m,n$, the inclusion 
$\ZZ/m \hookrightarrow \ZZ/{mn}$
induces a homomorphism on motivic cohomology denoted
$\inc^{m}_{mn}$, and the projection
$\ZZ/mn \twoheadrightarrow \ZZ/n$
induces a homomorphism on motivic cohomology denoted
$\pr^{mn}_{n}$. The short exact sequence
\[ 0\to \ZZ/m \to \ZZ/mn \to \ZZ/n \to 0\]
induces a Bockstein or boundary
homomorphism $\partial^{n}_{m}$. 

\begin{theorem}\label{thm:slices-kqnh}
  Let $0<n,0\leq q$.
  The slices of $\kq_{n\hyper}$ are given as follows:
  \begin{align*}
    \s_{2q} (\kq_{n\hyper}) & = \Sigma^{4q,2q} \MZ/2n \vee \bigvee_{j=0}^{2q-1} \Sigma^{2q+j,2q}\MZ/2 \\
    \s_{2q+1} (\kq_{n\hyper}) & = \s_{2q+1}(\kq)\vee \Sigma^{1,0}\s_{2q+1}(\kq) =
                        \bigvee_{j=0}^{2q+1} \Sigma^{2q+1+j,2q+1}\MZ/2 
  \end{align*}
  The canonical map $\kq_{n\hyper} \to \ckw_{n\hyper}$ induces the following
  map on summands of slices: 
  \begin{align*}
    \Sigma^{4q,2q}\MZ/2n & \xrightarrow{(\partial^{2n}_{2},\pr^{2n}_{2})}
    \Sigma^{4q+1,2q}\MZ/2\vee \Sigma^{4q,2q}\MZ/2 \\
    \Sigma^{q+j,q}\MZ/2 & \xrightarrow{(1)} \Sigma^{q+j,q}\MZ/2
                          \quad \mathrm{if}\   j \equiv 0(2) \  \mathrm{or} \  n\equiv 0(2) \\
    \Sigma^{q+j,q}\MZ/2 & \xrightarrow{(\Sq^1,1)}
    \Sigma^{q+j+1,q}\MZ/2\vee \Sigma^{q+j,q}\MZ/2 \quad \mathrm{if}\ j \equiv 1(2) \ \mathrm{and} \  n \equiv 1(2)
  \end{align*}
  The restriction of the slice $\dd^{1}$-differential for $\kq_{n\hyper}$
  to the summands $\Sigma^{q+j,q}\MZ/2$ or $\Sigma^{4q,2q}\MZ/2n$ of $\s_{q}(\kq_{n\hyper})$ for $n$ even is given by
  \begin{align*}
    \dd^{1}(\kq_{n\hyper})(q,j) 
    & =  
      \begin{cases}
        (\Sq^{3}\Sq^{1},0,\Sq^{2}) & q-1>j\equiv 0,1\bmod 4 \\
        (\Sq^{3}\Sq^{1},0,\Sq^{2}+\rho\Sq^{1},0,\tau) & q-1>j\equiv 2,3\bmod 4 \\
      \end{cases} \\
    \dd^{1}(\kq_{n\hyper})(q,q-1) 
    & =  
      \begin{cases}
        (\partial^{2}_{2n}\Sq^{2}\Sq^{1},0,\Sq^{2}) & q-1\equiv 0\bmod 4 \\
        (\Sq^{3}\Sq^{1},0,\Sq^{2})& q-1\equiv 1\bmod 4 \\
        (\partial^{2}_{2n}\Sq^{2}\Sq^{1},0,\Sq^{2}+\rho\Sq^1,0,\tau) & q-1\equiv 2\bmod 4 \\
        (\Sq^{3}\Sq^{1},0,\Sq^{2}+\rho\Sq^{1},0,\tau) & q-1\equiv 3\bmod 4. 
      \end{cases} \\
    \dd^{1}(\kq_{n\hyper})(q,q) 
    & =  
      \begin{cases}
        (\Sq^{2}\partial^{2n}_{2},\Sq^{2}\pr^{2n}_{2}) & q\equiv 0\bmod 4 \\
        (\inc^{2}_{2n}\Sq^{2}\Sq^{1},\Sq^{2}) & q\equiv 1\bmod 4 \\
        (\Sq^{2}\partial^{2n}_{2},\Sq^{2}\pr^{2n}_{2},\tau\partial^{2n}_{2},\tau\pr^{2n}_{2}) & q\equiv 2\bmod 4 \\
        (\inc^{2}_{2n}\Sq^{2}\Sq^{1},\Sq^{2}+\rho\Sq^1,0,\tau) & q\equiv 3\bmod 4. 
      \end{cases} 
  \end{align*}
  The restriction of the slice $\dd^{1}$-differential for $\kq_{n\hyper}$ to the summands $\Sigma^{q+j,q}\MZ/2$ or $\Sigma^{4q,2q}\MZ/2n$ of $\s_{q}(\kq_{n\hyper})$ for $n$ odd is given by
  \begin{align*}
    \dd^{1}(\kq_{n\hyper})(q,j) 
    & =  
      \begin{cases}
        (\Sq^{3}\Sq^{1},0,\Sq^{2}) & q-1>j\equiv 0\bmod 4 \\
        (\Sq^{3}\Sq^{1},\Q_1,\Sq^{2},\tau\Sq^{1}) & q-1>j\equiv 1\bmod 4 \\
        (\Sq^{3}\Sq^{1},0,\Sq^{2}+\rho\Sq^{1},0,\tau) & q-1>j\equiv 2\bmod 4 \\
        (\Sq^{3}\Sq^{1},\Q_1,\Sq^{2}+\rho\Sq^{1},\tau\Sq^{1}+\rho,\tau) & q-1>j\equiv 3\bmod 4, 
      \end{cases} \\
    \dd^{1}(\kq_{n\hyper})(q,q-1) 
    & =  
      \begin{cases}
        (\partial^{2}_{2n}\Sq^{2}\Sq^{1},0,\Sq^{2}) & q-1\equiv 0\bmod 4 \\
        (\Sq^{3}\Sq^{1},\Q_1,\Sq^{2},\tau\Sq^1) & q-1\equiv 1\bmod 4 \\
        (\partial^{2}_{2n}\Sq^{2}\Sq^{1},0,\Sq^{2}+\rho\Sq^{1},0,\tau) & q-1\equiv 2\bmod 4 \\
        (\Sq^{3}\Sq^{1},\Q_1,\Sq^{2}+\rho\Sq^1,\tau\Sq^1+\rho,\tau) & q-1\equiv 3\bmod 4. 
      \end{cases} \\
    \dd^{1}(\kq_{n\hyper})(q,q) 
    & =  
      \begin{cases}
        (\Sq^{2}\partial^{2n}_{2},\Sq^{2}\pr^{2n}_{2}) & q\equiv 0\bmod 4 \\
        (\inc^{2}_{2n}\Sq^{2}\Sq^{1}+\partial^{2}_{2n}\Sq^2,\Sq^{2},\tau\Sq^1) & q\equiv 1\bmod 4 \\
        (\Sq^{2}\partial^{2n}_{2},\Sq^{2}\pr^{2n}_{2},\tau\partial^{2n}_{2},\tau\pr^{2n}_{2}) & q\equiv 2\bmod 4 \\
        (\inc^{2}_{2n}\Sq^{2}\Sq^{1}+\partial^{2}_{2n}\Sq^2,\Sq^{2}+\rho\Sq^1,\tau\Sq^1+\rho,\tau) & q\equiv 3\bmod 4. 
      \end{cases} 
  \end{align*}
\end{theorem}

\begin{proof}
  The description of the slices follows from the fact that the slice functors
  are triangulated, and \cite[Theorem 3.2]{aro.kq}. 
  The effect of the canonical
  map $\kq_{n\hyper}\to \ckw_{n\hyper}$ on slices is readily obtained,
  except for the occurrence of $\Sq^1$. The latter follows from the
  determination of the first slice differential for $\ckw_{n\hyper}=\ckw\vee
  \Sigma^{1,0}\ckw$, compared with possible first slice differentials for
  $\kq_{n\hyper}$ compatible with the first slice differential for $\kq$,
  as described in \cite[Theorem 3.5]{aro.kq}. Determining the first slice differential
  for $\kq_{n\hyper}$ is then essentially straightforward.
\end{proof}

Let $1<n\in \NN$, set $\kq_{n}:=\kq\smash \moore{n}$,
and $\ckw_{n}:=\ckw\smash \moore{n}$.
For comparison, consider the form of the first slice differential
for $\ckw_{2}$, as obtained in \cite[Theorem 4.3, Theorem 4.14]{kro.hermitian}.
Despite the abstract isomorphisms $\s_\ast\ckw_{2}\iso \s_\ast\ckw_{\hyper}$
and $\s_\ast \kq_{2}\iso \s_\ast\kq_{\hyper}$, 
the first slice differentials differ.
Note that if $n$ is odd, then all slices of $\ckw_{n}$ vanish,
although $\ckw_{n}$ itself does not over formally real fields. 

\begin{theorem}
  \label{thm:ckw2-diff}
  The restriction of the slice $\dd^{1}$-differential for $\ckw_{2}$
  to the summand $\Sigma^{q + j,q}\MZ/2$ of $\s_{q}(\ckw_{2})$ is given by 
  \begin{equation*}
    \dd^{1}(\ckw_{2})(q,j)
    = 
    \begin{cases}
      (\Sq^{3}\Sq^{1},0,\Sq^{2}) & j\equiv 0\bmod 4 \\
      (\Sq^{3}\Sq^{1},\Q_1,\Sq^{2},\rho+\tau\Sq^{1}) & j\equiv 1\bmod 4 \\
      (\Sq^{3}\Sq^{1},0,\Sq^{2}+\rho\Sq^{1},0,\tau) & j\equiv 2\bmod 4 \\
      (\Sq^{3}\Sq^{1},\Q_1,\Sq^{2}+\rho\Sq^{1},\tau\Sq^{1},\tau) & j\equiv 3\bmod 4. 
    \end{cases}
    \label{equation:KW2-diff1}
  \end{equation*}
  Here the $i$th component of $\dd^{1}(\ckw_{2})(q,j)$ is a map $\Sigma^{q+j,q}\MZ/2\to\Sigma^{q+j+i,q+1}\MZ/2$. 
\end{theorem}

\begin{theorem}
  \label{thm:diff-kq-mod2}
  The restriction of the slice $\dd^{1}$-differential for $\kq_{2}$
  to the summand $\Sigma^{q+j,q}\MZ/2$ of $\s_{q}(\kq_{2})$ is given by
  \begin{align*}
    \dd^{1}(\kq_{2})(q,j) 
    & =  
      \begin{cases}
        (\Sq^{3}\Sq^{1},0,\Sq^{2}) & q>j\equiv 0\bmod 4 \\
        (\Sq^{3}\Sq^{1},\Q_1,\Sq^{2},\rho+\tau\Sq^{1}) & q>j\equiv 1\bmod 4 \\
        (\Sq^{3}\Sq^{1},0,\Sq^{2}+\rho\Sq^{1},0,\tau) & q>j\equiv 2\bmod 4 \\
        (\Sq^{3}\Sq^{1},\Q_1,\Sq^{2}+\rho\Sq^{1},\tau\Sq^{1},\tau) & q>j\equiv 3\bmod 4, 
      \end{cases} \\
    \dd^{1}(\kq_{2})(q,q) & =  
                            \begin{cases}
                              (\Sq^{2}\Sq^{1},\Sq^{2}+\rho\Sq^{1}) & q\equiv 0\bmod 4 \\
                              (\Q_1,\Sq^{2},\rho+\tau\Sq^{1},0) & q\equiv 1\bmod 4 \\
                              (\Sq^{2}\Sq^{1},\Sq^{2}+\rho\Sq^{1},\tau\Sq^{1},\tau) & q\equiv 2\bmod 4 \\
                              (\Q_1,\Sq^{2}+\rho\Sq^{1},\tau\Sq^{1},\tau) & q\equiv 3\bmod 4. 
                            \end{cases} 
  \end{align*}
  Here the $i$th component of $\dd^{1}(\kq_{2})(q,j)$ is a map $\Sigma^{q+j,q}\MZ/2\to\Sigma^{q+j+i,q+1}\MZ/2$. 
\end{theorem}

Note that \cite[Theorems 4.24, 4.36]{kro.hermitian} contain information about
first slice differentials for
$\kq_{2^r}$ and $\ckw_{2^r}$, which could be used to analyze the first
slice differentials for $\moore{2^r}$.

\begin{theorem}\label{thm:slices-moorenh}
  Let $n>0$, and let $g:=\gcd(2n,12)$. 
  The slices of $\moore{n\hyper}$ are given as follows, up
  to summands of simplicial suspension higher than 
  $q+2$ for $q\geq 4$.
  \begin{align*}
    \s_0 (\moore{n\hyper}) & = \s_0(\unit)/2n \\
    \s_1 (\moore{n\hyper}) & = \s_1(\unit)\vee \Sigma^{1,0}\s_1\unit \\
    \s_2 (\moore{n\hyper}) & = \Sigma^{2,2}\MZ/2\{\alpha_1^2\}\vee \Sigma^{3,2}\MZ/2\{\overline{\alpha_1^2}\} \vee \Sigma^{3,2}\MZ/g\{\alpha_2\}\vee \Sigma^{4,2}\MZ/g\{\overline{\alpha_2}\}  \\
    \s_3 (\moore{n\hyper}) & = \s_3(\unit)\vee \Sigma^{1,0}\s_3\unit \\
    \s_4 (\moore{n\hyper}) & = \Sigma^{4,4}\MZ/2\{\alpha_1^4\}\vee \Sigma^{5,4}\MZ/2\{\overline{\alpha_1^4}\} \vee \Sigma^{6,4}\MZ/2\{\alpha_1\alpha_3\} \vee \Sigma^{6,4}\MZ/2\{\nu^2\} \vee \dotsm   \\
    \s_{q} (\moore{n\hyper}) & = \Sigma^{q,q}\MZ/2\{\alpha_1^q\}\vee \Sigma^{q+1,q}\MZ/2\{\overline{\alpha_1^q}\}\vee \Sigma^{q+2,q}\MZ/2\{\alpha_1^{q-3}\alpha_3\}\vee \dotsm 
  \end{align*}
  The unit map $\unit\to \kq$ induces a map $\moore{nh}\to \kq_{nh}$,
  which induces the identity map on the slice summands 
  $\MZ/2n,\Sigma^{q,q}\MZ/2\{\alpha_1^q\}$, 
  $\Sigma^{q+1,q}\MZ/2\{\overline{\alpha_1^q}\}$,
  $\Sigma^{q+2,q}\MZ/2\{\alpha_1^{q-3}\alpha_3\}$,
  and $\Sigma^{q+3,q}\MZ/2\{\overline{\alpha_1^{q-3}\alpha_3}\}$, and the map
  \[ \Sigma^{3,2}\MZ/g\vee \Sigma^{4,2}\MZ/g \xrightarrow{\begin{pmatrix} \partial^g_{2n} & \inc^{g}_{2n} \end{pmatrix}} \Sigma^{4,2}\MZ/2n\]
  on summands of the two-slices. 
\end{theorem}

\begin{proof}
  The form of the slices follows from 
  \cite[Section 8]{levine.comparison} or \cite[Theorem 2.12]{rso.oneline}, 
  since $\hyper$ induces multiplication
  by $2$ on slices, and the slice functors are triangulated.
  The map $\moore{n\hyper}\to\kq_{n\hyper}$ 
  is determined by the unit map $\unit\to \kq$, whose behaviour on
  slices can be read off from \cite[Lemmas 2.28, 2.29]{rso.oneline}.
\end{proof}

\begin{theorem}\label{thm:slice-diff-moorenh}
  The first slice differential has the following
  form on the given summands:
  \begin{align*}
    \dd^1(\moore{n\hyper})(\alpha_1^0) & = (\Sq^2\partial^{2n}_{2},\Sq^2\pr^{2n}_{2})  \\
  \dd^1(\moore{n\hyper})(\alpha_1) & = 
    \begin{cases} 
      (\partial^{2}_{g}\Sq^2\Sq^1,\inc^2_{g}\Sq^2\Sq^1,0,\Sq^2) & n\equiv 0(4) \\
      (0,\inc^2_{g}\Sq^2\Sq^1,0,\Sq^2) & n\equiv 2(4) \\
      (\partial^{2}_{g}\Sq^2\Sq^1,0,0,\Sq^2) & n\equiv 1(2)
    \end{cases} \\
    \dd^1(\moore{n\hyper})(\alpha_1^{q}) & = (\Sq^3\Sq^1,0,\Sq^2) \quad q\geq 2\\
    \dd^1(\moore{n\hyper})(\overline{\alpha_1}) & = 
    \begin{cases} 
      (\inc^2_{g}\Sq^2\Sq^1,0,\Sq^2) & n\equiv 0(2) \\
      (\inc^{2}_{g}\Sq^2\Sq^1,\inc^{2}_{g}\Sq^2,\Sq^2,\tau\Sq^1) & n\equiv 1(2)
    \end{cases} \\
    \dd^1(\moore{n\hyper})(\overline{\alpha^{q}_1}) & = 
    \begin{cases} 
      (\Sq^3\Sq^1,0,\Sq^2) & n\equiv 0(2) \\
      (\Sq^3\Sq^1,\Q_1,\Sq^2,\tau\Sq^1) & n\equiv 1(2)
    \end{cases} \\
    \dd^1(\moore{n\hyper})(\alpha_{2}) & =  
      (0,\Sq^2\partial^g_2,0,\tau\partial^g_2)  \\
    \dd^1(\moore{n\hyper})(\overline{\alpha_{2}}) & =  
    \begin{cases} 
      (\Sq^2\partial^g_2,0,\tau\partial^g_2,0) & n\equiv 0(4) \\
      (\Sq^2\partial^g_2,\Sq^2\pr^g_2,\tau\partial^g_2,\tau\pr^g_2) & n\not\equiv 0(4)
    \end{cases} \\
    \dd^1(\moore{n\hyper})({\alpha^{q-3}_1\alpha_{3}}) & =  
      (\Sq^3\Sq^1,0,\Sq^2+\rho\Sq^1,0,\tau) \quad q\geq 3 \\
    \dd^1(\moore{n\hyper})(\overline{\alpha^{q-3}_1\alpha_{3}}) & = 
    \begin{cases} 
      (\Sq^3\Sq^1,0,\Sq^2+\rho\Sq^1,0,\tau) & n\equiv 0(2) \\
      (\Sq^3\Sq^1,\Q_1,\Sq^2+\rho\Sq^1,\tau\Sq^1+\rho,\tau) & n\equiv 1(2)
    \end{cases} 
  \end{align*}
\end{theorem}

\begin{proof}
  Based on the form of the unit map, most parts 
  of the first slice differential are determined
  by Theorem~\ref{thm:slices-kqnh}. More precisely,
  the differentials $\s_{0}\moore{n\hyper}\to \Sigma^{1,0}\s_{1}\moore{n\hyper}$,
  $\s_{2}\moore{n\hyper}\to \Sigma^{1,0}\s_{3}\moore{n\hyper}$
  and the listed behaviour for $q\geq 2$ is determined by 
  Theorem~\ref{thm:slices-kqnh}. The remaining identities
  follow from Adem relations and the property $\dd^2=0$, as
  well as comparison with $\dd^1(\unit)$.
\end{proof}

Theorem~\ref{thm:slice-diff-moorenh} may be used for computations
of $\pi_{1+\bideg}\moore{n\hyper}$ and $\pi_{2+\bideg}\moore{n\hyper}$.
Compared with corresponding slice spectral
sequence computations for $\unit$, the absence of integral
motivic cohomology groups may be viewed as an advantage. 
Concrete presentations
for $H^{\star-k,\star}$ as $\KMW$-modules seem to be known only in
very few cases, contrary to the $\KMW$-module 
$h^{\star-k,\star}\iso \KMW_{\star -k}/(\eta,2)$.
\providecommand{\bysame}{\leavevmode\hbox to3em{\hrulefill}\thinspace}

\providecommand{\MR}{\relax\ifhmode\unskip\space\fi MR }
\providecommand{\MRhref}[2]{%
  \href{http://www.ams.org/mathscinet-getitem?mr=#1}{#2}
}
\providecommand{\href}[2]{#2}

\end{document}